\documentclass[12pt,reqno]{amsart}
\usepackage{amsmath}
\usepackage{amssymb}
\usepackage{algorithm}
\usepackage{algorithmic}

\def\B'c{{\mathcal{B'}}}
\def\U'c{{\mathcal{U'}}}

\def\opn#1#2{\def#1{\operatorname{#2}}}
\opn\chara{char}
\opn\length{\ell}
\opn\projdim{proj\,dim}
\opn\injdim{inj\,dim}
\opn\ini{in}
\opn\rank{rank}
\opn\depth{depth}
\opn\height{ht}
\opn\embdim{emb\,dim}
\opn\codim{codim}

\opn\Tr{Tr}
\opn\bigrank{big\,rank}
\opn\superheight{superheight}\opn\lcm{lcm}
\opn\trdeg{tr\,deg}\opn\reg{reg}
\opn\lreg{lreg}
\opn\set{set}
\opn\supp{Supp}
\opn\shad{Shad}
\opn\del{del}
\opn\div{div}
\opn\Div{Div}
\opn\cl{cl}
\opn\Cl{Cl}
\opn\Spec{Spec}
\opn\Supp{Supp}
\opn\supp{supp}
\opn\Sing{Sing}
\opn\Ass{Ass}
\opn\Ann{Ann}
\opn\Rad{Rad}
\opn\Soc{Soc}
\opn\Ker{Ker}
\opn\Coker{Coker}
\opn\Im{Im}
\opn\Hom{Hom}
\opn\Tor{Tor}
\opn\Ext{Ext}
\opn\End{End}
\opn\Aut{Aut}
\opn\id{id}

\opn\nat{nat}
\opn\GL{GL}
\opn\SL{SL}
\opn\mod{mod}
\opn\ord{ord}
\opn\aff{aff}
\opn\con{conv}
\opn\relint{relint}
\opn\st{st}
\opn\lk{lk}
\opn\cn{cn}
\opn\core{core}
\opn\vol{vol}
\opn\gr{gr}

\def\pot#1#2{#1[\kern-0.28ex[#2]\kern-0.28ex]}

\opn\dirlim{\underrightarrow{\lim}}
\opn\invlim{\underleftarrow{\lim}}

\def\pnt{{\raise0.5mm\hbox{\large\bf.}}}

\def\Implies{\ifmmode\Longrightarrow \else
     \unskip${}\Longrightarrow{}$\ignorespaces\fi}
\def\implies{\ifmmode\Rightarrow \else
     \unskip${}\Rightarrow{}$\ignorespaces\fi}
\def\iff{\ifmmode\Longleftrightarrow \else
     \unskip${}\Longleftrightarrow{}$\ignorespaces\fi}

\let\:=\colon
\newtheorem{Theorem}{Theorem}[section]
\newtheorem{Lemma}[Theorem]{Lemma}
\newtheorem{Corollary}[Theorem]{Corollary}
\newtheorem{Proposition}[Theorem]{Proposition}
\newtheorem{Remark}[Theorem]{Remark}

\let\epsilon=\varepsilon
\let\phi=\varphi
\let\kappa=\varkappa
\begin{document}
\title{On the Hilbert series of vertex cover algebras of Cohen-Macaulay
bipartite graphs}
\author{Cristian Ion}
\address{Faculty of Mathematics and Computer Science, Ovidius University,
Bd.\ Mamaia 124, 900527 Constanta, Romania,}
\email{cristian.adrian.ion@gmail.com}
\maketitle

\begin{abstract}
We study the Hilbert function and the Hilbert series of the vertex cover
algebra $A(G)$, where $G$ is a Cohen-Macaulay bipartite graph.\\


\end{abstract}

\section{Introduction}

Let $G=(V,E)$ be a simple (i.e., finite, undirected, loop less and without
multiple edges) graph with the vertex set $V=[n]$ and the edge set $E=E(G)$.
A\textit{\ vertex cover} of $G$ is a subset $C\subset V$ such that $C\cap
\{i,j\}\neq \emptyset $, for any edge $\{i,j\}\in E(G)$. A vertex cover $C$
of $G$ is called \textit{minimal} if no proper subset $C^{\prime }\subset C$
is a vertex cover of $G$. A graph $G$ is called \textit{unmixed} if all
minimal vertex covers of $G$ have the same cardinality. Let $%
R=K[x_{1},\ldots ,x_{n}]$ be the polynomial ring in $n$ variables over a
field $K$. The \textit{edge ideal} of $G$ is the monomial ideal $I(G)$ of $R$
generated by all the quadratic monomials $x_{i}x_{j}$ with $\{i,j\}\in E(G)$%
. It is said that a graph $G$ is $Cohen-Macaulay$ (over $K$) if the quotient
ring $R/I(G)$ is Cohen-Macaulay. Every Cohen-Macaulay graph is unmixed.

A vertex cover $C\subset \lbrack n]$ can be represented as a $(0,1)-$vector $%
c$ that satisfies the restriction $c(i)+c(j)\geq 1$, for every $\{i,j\}\in
E(G)$. For each $k\in \mathbf{N}$, a\textit{\ vertex cover }of\textit{\ }$G$
of \textit{order} $k$, or simply a $k-$\textit{vertex cover} of $G$, is a
vector $c\in \mathbf{N}^{n}$ such that $c(i)+c(j)\geq k$, for every $%
\{i,j\}\in E(G)$. The \textit{vertex cover algebra} $A(G)$ is defined as the
subalgebra of the one variable polynomial ring $R[t]$ generated by all
monomials $x_{1}^{c_{1}}\cdots x_{n}^{c_{n}}t^{k}$, where $c=(c_{1},\ldots
,c_{n})\in \mathbf{N}^{n}$ is a $k$-vertex cover of $G$. This algebra was
introduced and first studied in \cite{HerHib3}. Let $\mathfrak{m}$ be the
maximal graded ideal of $R$. The graded $K$-algebra $\bar{A}(G)=A(G)/%
\mathfrak{m}A(G)$ is called the \textit{basic cover algebra }and it was
introduced and first studied in \cite[Section 3]{HerHib2}.

Our aim in this paper is to study the Hilbert function and series of the
vertex cover algebra $A(G)$ for Cohen-Macaulay bipartite graphs.

Let $P_{n}=\{p_{1},p_{2},\ldots ,p_{n}\}$ be a poset with a partial order $%
\leq $. Let $G=G(P_{n})$ be the bipartite graph on the set $V_{n}=W\cup
W^{\prime }$, where $W=\{x_{1},x_{2},\ldots ,x_{n}\}$ and $W^{\prime
}=\{y_{1},y_{2},\ldots ,y_{n}\}$, whose edge set $E(G)$ consists of all $2$%
-element subsets $\{x_{i},y_{j}\}$ with $p_{i}\leq p_{j}$. It is said that a
bipartite graph $G$ on $V_{n}=W\cup W^{\prime }$ \textit{comes from a poset}%
, if there exists a finite poset $P_{n}$ on $\{p_{1},p_{2},\ldots ,p_{n}\}$
such that $p_{i}\leq p_{j}$ implies $i\leq j$, and after relabeling of the
vertices of $G$ one has $G=G(P_{n})$. Herzog and Hibi proved in \cite%
{HerHib1} that a bipartite graph $G$ is Cohen-Macaulay if and only if $G$
comes from a poset.

In Section 2, we firstly notice that the Hilbert function and series of the
vertex cover algebras $A(G)$ are invariant to poset isomorphisms. We obtain
a recurrence relation for the minimal vertex covers of a Cohen-Macaulay
graph $G$ and we study the Hilbert function of $A(G).$

In Section 3, we study the Hilbert series of $A(G)$. For a poset $%
P_{n}=\{p_{1},p_{2},...,p_{n}\}$ we denote by $\mathcal{J}(P_{n})$ the
lattice of all poset ideals of $P_{n}$. For each subset $\emptyset \neq
F\subset \lbrack n]$ we denote by $P_{n}(F)$ the subposet of $P_{n}$ induced
by the subset $\{p_{i}|i\in F\}$ and by $G_{F}$ the bipartite graph that
comes from $P_{n}(F)$. The main result of this paper is given in Theorem \ref%
{comp}, which shows that one may reduce the computation of the Hilbert
series of the vertex cover algebra $A(G)$ to the computation of the Hilbert
series of the basic cover algebra $\bar{A}(G_{F})$, for all $F\subset
\lbrack n]$. If $F=\emptyset $, then, by convention, the Hilbert series of $%
\bar{A}(G_{F})$ is equal to $\frac{1}{1-z}$. Namely, we have the following
formula: 
\begin{equation*}
H_{A(G)}(z)=\frac{1}{(1-z)^{n}}\sum\limits_{F\subset \lbrack n]}H_{\bar{A}%
(G_{F})}(z)\left( \frac{z}{1-z}\right) ^{n-\left\vert F\right\vert }\text{.}
\end{equation*}

Moreover, we give a combinatorial interpretation for the $h$-vector of $A(G)$
in terms of the poset $P_n.$ Using this interpretation we show that the $h$%
-vector of $A(G)$ is unimodal. We give bounds for its components and derive
bounds for $e(A(G)),$ the multiplicity of $A(G).$

We show that both chains and antichains are uniquely determined up to a
poset isomorphism by the Hilbert series of their corresponding vertex cover
algebras.

\section{Vertex cover algebras of Cohen-Macaulay bipartite graphs}

Let $S=K[x_{1},...,x_{n},y_{1},...,y_{n}]$ and let $G=G(P_{n})$, where $%
P_{n}=\{p_{1},...,p_{n}\}$ is a poset with a partial order $\leq $. We
recall that, by \cite{HerHib3}, the vertex cover algebra $A(G)$ is standard
graded over $S$ and it is the Rees algebra of the \textit{cover ideal} $%
I_{G} $, which is generated by all monomials $x_{1}^{c_{1}}\cdots
x_{n}^{c_{n}}y_{1}^{c_{n+1}}...y_{n}^{c_{2n}}$, where $c=(c_{1},\ldots
,c_{2n})$ is a $1$-vertex cover of $G$. Thus

\begin{center}
$A(G)=S\oplus I_{G}t\oplus \ldots \oplus I_{G}^{k}t^{k}\oplus \ldots $
\end{center}

Let $\{m_{1},m_{2},...,m_{l}\}$ be the minimal system of generators of $%
I_{G} $. We view $A(G)$ as a standard graded $K$-algebra by assigning to
each $x_{i}$ and $y_{j}$, $1\leq i,j\leq n$ and to each $m_{k}t$, $1\leq
k\leq l$, the degree $1$. Since each monomial $m_{k\text{ }}$corresponds to
a minimal vertex cover of $G$ of cardinality $n$, the Hilbert function of $%
A(G)$ is given by 
\begin{equation}
H(A(G),k)=\sum\limits_{j=0}^{k}\dim _{K}(I_{G}^{j})_{jn+(k-j)}\text{, for
all }k\geq 0\text{.}  \label{equ1}
\end{equation}

\begin{Remark}
\label{iso} \emph{Let $P_{n}=\{p_{1},...,p_{n}\}$ and $P_{n}^{\prime
}=\{p_{1}^{\prime },...,p_{n}^{\prime }\}$ be two isomorphic finite posets
and let $G=G(P_{n})$ and $G^{\prime }=G(P_{n}^{\prime })$. Then the cover
ideals $I_{G}$ and $I_{G^{\prime }}$ are isomorphic as graded $K$-vector
spaces and, consequently, the Hilbert function and series of $A(G)$ and $%
A(G^{\prime })$ coincide. Let $f:P_{n}\rightarrow P_{n}^{\prime }$ be a
poset isomorphism (i.e., $f $ is a bijective map with $p_{i}\leq p_{j}$ if
and only if $f(p_{i})\leq f(p_{j})$)$.$ Then $f$ induces a permutation $g$
of $[n]$, $i\rightarrow g(i) $, defined by $p_{g(i)}^{\prime }=f(p_{i})$,
for every $i\in \lbrack n]. $ We notice that 
\begin{equation}
p_{i}\leq p_{j}\Leftrightarrow f(p_{i})\leq f(p_{j})\Leftrightarrow
p_{g(i)}^{\prime }\leq p_{g(j)}^{\prime },  \label{equ2}
\end{equation}
and we define a map $h:V(G)\rightarrow V(G^{\prime })$ as follows: 
\begin{eqnarray*}
h(x_{i}) &=&x_{g(i)},\text{ if }i\in \lbrack n]\text{,} \\
h(y_{j}) &=&y_{g(j)},\text{ if }j\in \lbrack n].
\end{eqnarray*}
Then $h$ induces a $K$-automorphism of $S$ which maps $I_{G}$ onto $%
I_{G^{\prime }}$, hence, $I_{G}$ and $I_{G^{\prime }}$ are isomorphic as
graded $K$-vector spaces. By (\ref{equ1}), we also have 
\begin{equation*}
H(A(G^{\prime }),k)=\sum\limits_{j=0}^{k}\dim _{K}(I_{G^{\prime
}}^{j})_{jn+(k-j)}\text{, for all }k\geq 0\text{.}
\end{equation*}
Since the powers $I_{G}^{j}$ and $I_{G^{\prime }}^{j}$ are isomorphic as
graded $K$-vector spaces as well, for all $j\geq 1$, we get $%
H(A(G),k)=H(A(G^{\prime }),k)$, for all $k\geq 0$.}
\end{Remark}

We denote by $\mathcal{M}(G)$ the set of minimal vertex covers of a graph $G$%
. Vertex covers and stable sets of a graph $G$ are dual concepts, that is, a
subset $C\subset V(G)$ is a vertex cover of $G$ if and only if the
complement set $V(G)\backslash C$ is a stable set of $G$ (\cite{VanVil}).
Next, inspired by \cite[Lemma 2.5]{VanVil}, we give a recurrence relation to
obtain the set of the minimal vertex covers of a Cohen-Macaulay bipartite graph $G_n$ which comes from a poset 
$P_n=\{p_1,\ldots,p_n\}.$ We denote by $G_{n-1}$ the subgraph of $G_n$ which comes from the poset $P_{n-1}=\{p_1,\ldots,p_{n-1}\}$
and by $V_{n-1}$ the set $\{x_1,\ldots,x_{n-1}\}\cup \{y_1,\ldots,y_{n-1}\}.$

\begin{Proposition}
\label{algo} Let $G_{n}=G(P_{n})$, where $P_{n}=\{p_{1},\ldots ,p_{n}\}$, $%
n\geq 2$, is a poset such that $p_{i}\leq p_{j}$ implies $i\leq j$. Then a
subset $C_{n}\subset V_{n}$ is a minimal vertex cover of $G_{n}$ if and only
if either $C_{n}=C_{n-1}\cup \{y_{n}\}$, where $C_{n-1}\subset V_{n-1}$ is a
minimal vertex cover of $G_{n-1}$ or $C_{n}=C_{n-1}\cup \{x_{n}\}$, where $%
C_{n-1}\subset V_{n-1}$ is a minimal vertex cover of $G_{n-1}$ such that $%
x_{i}\in C_{n-1}$ for each $i\in \lbrack n-1]$ with $p_{i}\leq p_{n}$.
\end{Proposition}

\begin{proof}
'If' it is straightforward.

Let us proof 'Only if'. Since $G_{n}$ is a Cohen-Macaulay graph, it is
unmixed and all its minimal vertex covers have the same cardinality, namely $%
n$, for every $n\geq 2$.

If $n=2$ the statement obviously holds.

We assume that $n\geq 3$. Let $C_{n}=\{c_{1},\ldots ,c_{n}\}$ be a minimal
vertex cover of $G_{n}$. Put $C_{n}=\{c_{1},\ldots ,c_{n}\}$, $%
C_{n-1}=C_{n}\cap V_{n-1}$ and $C_{n}^{\prime }=C_{n}\cap \{x_{n},y_{n}\}$.
Obviously, $\left\vert C_{n}^{\prime }\right\vert \leq 2$.

If $\left\vert C^{\prime}_n\right\vert =0,$ then $C_n\cap\{x_n,y_n\}=\emptyset,$ which is impossible. 
Now let us suppose that $\left\vert C_{n}^{\prime }\right\vert =2$, hence $%
C_{n}^{\prime }=\{x_{n},y_{n}\}$ and $\left\vert C_{n-1}\right\vert =n-2$.
Since $C_{n}$ is a vertex cover of $G_{n}$, it follows that the intersection
of $C_{n}$ with every edge $\{x_{i},y_{j}\}$ of the subgraph $G_{n-1}$ ($%
1\leq i\leq j\leq n-1$) is a nonempty subset of $C_{n-1}$, hence $C_{n-1}$
is a vertex cover of $G_{n-1}$ of cardinality $n-2.$ But this is impossible since 
 all minimal vertex covers of $G_{n-1}$ have
the cardinality equal to $n-1$. 

It follows that $\left\vert C_{n}^{\prime }\right\vert =1$, $\left\vert
C_{n-1}\right\vert =n-1$ and exactly one of the vertices $x_{n}$ or $y_{n}$
belongs to $C_{n}$. We can put, without loss of generality, either $%
c_{n}=x_{n}$ or $c_{n}=y_{n}$, and $C_{n-1}=\{c_{1},\ldots ,c_{n-1}\}\subset
V_{n-1}$. Since $C_{n}$ is a vertex cover of $G_{n-1}$, the intersection of $%
C_{n}$ with every edge $\{x_{i},y_{j}\}$ of the subgraph $G_{n-1}$ ($1\leq
i\leq j\leq n-1$) is a nonempty subset of $C_{n-1}$, hence $C_{n-1}$ is a
vertex cover of $G_{n-1}$. Moreover, $C_{n-1}$ is a minimal vertex cover of $%
G_{n-1}$, since $\left\vert C_{n-1}\right\vert $ $=n-1$.

If we choose $c_{n}=x_{n}$, then $y_{n}\notin C_{n}$. Since $C_{n}$ is a
vertex cover of $G_{n}$, it follows that $C_{n}\cap
\{x_{i},y_{n}\}=\{x_{i}\} $, for every $\{x_{i},y_{n}\}\in E(G_{n})$ with $%
i\in \lbrack n-1]$, which implies that $x_{i}\in C_{n}$, for each $i\in
\lbrack n-1]$ with $\{x_{i},y_{n}\}\in E(G_{n})$. Hence $x_{i}\in C_{n}\cap
V_{n-1}=C_{n-1}$, for each $i\in \lbrack n-1]$ with $p_{i}\leq p_{n}$.

If we choose $c_{n}=y_{n}$, then there is no (other) restriction on the
minimal vertex cover $C_{n-1}$ of $G_{n-1}$.
\end{proof}

\begin{Remark} Let $G$ be a Cohen-Macaulay bipartite graph which comes from the poset $P_n.$
\emph{\label{alg} By \cite[Theorem 2.1]{HerHib2} there is a one-to-one
correspondence between the set $\mathcal{M}(G)$ and the distributive lattice 
$\mathcal{J}(P_{n})$ of all poset ideals of $P_{n}$. Thus it can be assigned
to each minimal vertex cover $C$ of $G$ the poset ideal $\alpha _{C}$ of $%
P_{n}$ that is defined as $\alpha _{C}=\{p_{i}|x_{i}\in C\}$. Conversely, if 
$\alpha $ is a poset ideal of $P_{n}$, then the corresponding set $C_{\alpha
}=\{x_{i}|p_{i}\in \alpha \}\cup \{y_{j}|p_{j}\not\in \alpha \}$ is a
minimal vertex cover of $G.$ By Proposition \ref{algo}, one may give a
recursive procedure to compute the lattice $\mathcal{J}(P_{n})$. }
\end{Remark}

For $C\in \mathcal{M}(G)$ we denote $m_{C}=(\prod\limits_{x_{i}\in
C}x_{i})\cdot (\prod\limits_{y_{j}\in C}y_{j})$. If $G$ is unmixed, then
each $C\in \mathcal{M}(G)$ has exactly $n$ vertices, hence, deg $m_{C}=n$,
for all $C\in \mathcal{M}(G)$. The next result shows a property of monotony
of the Hilbert function of an unmixed bipartite graph.

\begin{Proposition}
\label{monotonous} Let $G$, $G^{\prime }$ and $G^{\prime \prime }$ be
unmixed bipartite graphs on $V_{n}$, $n\geq 1$, such that $E(G^{\prime
\prime })\subset E(G)\subset E(G^{\prime })$. Then the following
inequalities hold: 
\begin{equation*}
H(A(G^{\prime }),k)\leq H(A(G),k)\leq H(A(G^{\prime \prime }),k),\text{ for
all }k\geq 0.
\end{equation*}
\end{Proposition}

\begin{proof} 
It is known (\cite[Theorem 5.1.b]{HerHib3}) that $I_{G}=(m_{C}\ |\ C\in \mathcal{%
M}(G))$. Similarly, we have $I_{G^{\prime }}=(m_{C}\ |\ C\in \mathcal{M}%
(G^{\prime }))$ and $I_{G^{\prime \prime }}=(m_{C}\ |\ C\in \mathcal{M}%
(G^{\prime \prime }))$. It follows that all the cover ideals are generated in the same degree $n.$

From the inclusions between the edge sets and the hypothesis of unmixedness, we get $\mathcal{M}(G^{\prime })\subset 
\mathcal{M}(G)\subset \mathcal{M}(G^{\prime \prime }).$
Therefore, $I_{G^{\prime
}}\subset I_{G}\subset I_{G^{\prime \prime }}$. We also have
\begin{equation}
(I_{G^{\prime }}^{a})_{b}\subset (I_{G}^{a})_{b}\subset (I_{G^{\prime \prime
}}^{a})_{b}\text{,}  \label{equ3}
\end{equation}

for all integers $a\geq 1$ and $b\geq 0$, which, by (\ref{equ1}), implies
the desired inequalities.
\end{proof}

It is obvious that, for the Cohen-Macaulay bipartite graphs, the chain
provides the largest number of edges and the antichain the smallest number
of edges.

\begin{Corollary}
\label{hilf} Let $G$ be a Cohen-Macaulay bipartite graph on $V_{n}$, $n\geq
1 $. Then the following inequalities hold: 
\begin{equation}
H(A(G^{\prime }),k)\leq H(A(G),k)\leq H(A(G^{\prime \prime }),k),\text{ for
all }k\geq 0,  \label{equ4}
\end{equation}%
where $G^{\prime }$ and $G^{\prime \prime }$ are bipartite graphs on $V_{n}$
that come from a chain, respectively, an antichain with $n$ elements.
\end{Corollary}

\begin{proof}
Let $G$, $G^{\prime }$, respectively, $G^{\prime \prime }$ be graphs that
come from a poset $P_{n}=\{p_{1},..,p_{n}\}$, a chain $P_{n}^{\prime
}=\{p_{1}^{\prime },..,p_{n}^{\prime }\}$, respectively, an antichain $%
P_{n}^{\prime \prime }=\{p_{1}^{\prime \prime },..,p_{n}^{\prime \prime }\}$%
. By Remark \ref{iso} we may asume that $p_{i}\leq p_{j}$ and $%
p_{i}^{\prime }\leq p_{j}^{\prime }$ imply $i\leq j$. It is straightforward
to notice that $E(G^{\prime \prime })\subset E(G)\subset E(G^{\prime })$.
Therefore, by applying Proposition \ref{monotonous}, the desired
inequalities follow.
\end{proof}

The next result stresses a property of monotony for the multiplicity of the
vertex cover algebra for unmixed bipartite graphs.

\begin{Corollary}
\label{multineq} Let $G$, $G^{\prime }$and $G^{\prime \prime }$ be unmixed
bipartite graphs on $V_{n}$ such that $E(G^{\prime \prime })\subset
E(G)\subset E(G^{\prime })$. Then the following inequalities hold: 
\begin{equation*}
e(A(G^{\prime }))\leq e(A(G))\leq e(A(G^{\prime \prime })).
\end{equation*}
\end{Corollary}

\begin{proof}
By Proposition \ref{monotonous} we have $H(A(G^{\prime })),k)\leq
H(A(G),k)\leq H(A(G^{\prime \prime }),k)$, for all $k\geq 0$. Since $%
H(A(G),k)$, $H(A(G^{\prime }),k)$, respectively, $H(A(G^{\prime \prime }),k)$
are all polynomials of degree $2n$ (since $\dim A(G)=\dim S + 1=2n+1$ \cite%
{BruHer}) with the leading coefficients $\frac{e(A(G))}{(2n)!}$, $\frac{%
e(A(G^{\prime }))}{(2n)!}$, respectively, $\frac{e(A(G^{\prime \prime }))}{%
(2n)!}$, the conclusion follows.
\end{proof}

\section{The Hilbert series of vertex cover algebras of Cohen-Macaulay
bipartite graphs}

Let $S=K[x_{1},\ldots ,x_{n},y_{1},\ldots ,y_{n}]$ be the polynomial ring in 
$2n$ variables over a field $K$ and let $G=G(P_{n})$, where $%
P_{n}=\{p_{1},...,p_{n}\}$ is a poset such that $p_{i}\leq p_{j}$ implies $%
i\leq j$.

We denote $B_{G}=K[\{x_{i}\}_{1\leq i\leq n},\{y_{j}\}_{1\leq j\leq
n},\{u_{\alpha }\}_{\alpha \in \mathcal{J}(P_{n})}]$. The\textit{\ toric
ideal} $Q_{G}$ of $A(G)$ is the kernel of the surjective homomorphism $\phi
:B_{G}\rightarrow A(G)$ defined by $\phi (x_{i})=x_{i}$, $\phi (y_{j})=y_{j}$%
, $\phi (u_{\alpha })=m_{\alpha }t$, where $m_{\alpha
}=(\prod\limits_{p_{i}\in \alpha }x_{i})\cdot (\prod\limits_{p_{j}\not\in
\alpha }y_{j})$, $\alpha\in\mathcal{J}(P_n),$ are the minimal monomial generators of the cover ideal $I_G.$

Let $<_{lex}$ denote the lexicographic order on $K[\{x_{i}\}_{1\leq i\leq
n},\{y_{j}\}_{1\leq j\leq n}]$ induced by the ordering $%
x_{1}>..>x_{n}>y_{1}>..>y_{n}$ and $<^{\#}$ the reverse lexicographic order
on $K[\{u_{\alpha }\}_{\alpha \in \mathcal{J}(P_{n})}]$ induced by an
ordering of the variables $u_{\alpha }$'s such that $u_{\alpha }>u_{\beta }$
if $\beta \subset \alpha $ in $\mathcal{J}(P_{n})$. Let $<_{lex}^{\#}$ be
the monomial order on $B_{G}$ defined as the product of the monomial orders $%
<_{lex}$ and $<^{\#}$ from above. The reduced Gr\"{o}bner basis $\mathcal{G}$
of the toric ideal $Q_{G}$ of $A(G)$ with respect to the monomial order $%
<_{lex}^{\sharp }$ on $B_{G}$ was computed in \cite[Theorem 1.1]{HerHib1}:

\begin{center}
$\mathcal{G}=\{\underline{x_{j}u_{\alpha }}-y_{j}u_{\alpha \cup \{p_{j}\}},$ 
$j\in \lbrack n],\alpha \in \mathcal{J}(P_{n}),p_{j}\not\in \alpha ,\alpha
\cup \{p_{j}\}\in \mathcal{J}(P_{n}),$

$\underline{u_{\alpha }u_{\beta }}-u_{\alpha \cup \beta }u_{\alpha \cap
\beta },\alpha ,\beta \in \mathcal{J}(P_{n}),\alpha \not\subset \beta ,\beta
\not\subset \alpha \}$,
\end{center}

\noindent where the initial monomial of each binomial of $\mathcal{G}$ is
the first monomial.

Let $S_{G}=K[\{u_{\alpha }\}_{\alpha \in \mathcal{J}(P_{n})}]$ be the
polynomial ring in $\left\vert \mathcal{J}(P_{n})\right\vert $ variables
over $K$, let $\bar{A}(G)$ the basic vertex cover algebra and $\Delta (\mathcal{J%
}(P_{n}))$ the order complex of the lattice $(\mathcal{J}(P_{n}),\subset )$ whose vertices are the chains of $P_n$.
(We refer the reader to \cite{BenCon}, \cite[Section 3]{HerHib2} for the
definition and properties of the basic cover algebra associated to a graph
and \cite[\S 5.1]{BruHer} for the definition and properties of the order
complex of a poset.) The \textit{toric ideal} $\bar{Q}_{G}$ of $\bar{A}(G)$
is the kernel of the surjective homomorphism $\pi :S_{G}\rightarrow \bar{A}%
(G)$, $\pi (u_{\alpha })=m_{\alpha }$. The reduced Gr\"{o}bner basis $%
\mathcal{G}_{0}$ of $\bar{Q}_{G}$ with respect to $<^{\#}$ on $S_{G}$ was
computed in \cite[Theorem 3.1]{HerHib2}:

\begin{center}
$\mathcal{G}_{0}=\{\underline{u_{\alpha }u_{\beta }}-u_{\alpha \cup \beta
}u_{\alpha \cap \beta }|\alpha ,\beta \in \mathcal{J}(P_{n}),\alpha
\not\subset \beta ,\beta \not\subset \alpha \}$,
\end{center}

\noindent where the initial monomial of each binomial of $\mathcal{G}_{0}$
is the first monomial.

\begin{Proposition}
\label{vect} The graded $K$-algebra $\bar{A}(G)$ and the the order complex $%
\Delta (\mathcal{J}(P_{n}))$ have the same  $h$-vector.
\end{Proposition}

\begin{proof}
$\bar{Q}_{G}$ is a graded ideal (generated by binomials) and the initial
ideal $\ini_{<^{\#}}(\bar{Q}_{G})$ of the toric ideal $\bar{Q}_{G}$
coincides with the Stanley-Reisner ideal $I_{\Delta (\mathcal{J}(P_{n}))}$,
hence $S_{G}/\bar{Q}_{G}$ and $K[\Delta (\mathcal{J}(P_{n}))]$ have the same 
$h$-vector. Since $S_{G}/\bar{Q}_{G}\simeq \bar{A}(G)$ as graded $K$%
-algebras, the conclusion follows.
\end{proof}

\begin{Remark}
\label{euler} \emph{Since $\mathcal{J}(P_{n})$ is a full sublattice of the
Boolean lattice $\mathcal{L}_{n}$ on the set $\{p_{1},p_{2},\ldots ,p_{n}\}$ (\cite[Theorem 2.2.]{HerHib2}), it follows that 
$\dim \Delta (\mathcal{J}(P_{n}))=n$. Let $%
h=(h_{0},h_{1},...,h_{n+1})$ be the $h$-vector of $\Delta (\mathcal{J}%
(P_{n}))$ and $\bar{A}(G)$. As we noticed above, the basic vertex cover
algebra $\bar{A}(G)$ can be identified with the Hibi ring $S_{G}/\bar{Q}_{G}$%
, which arises from the distributive lattice $\mathcal{J}(P_{n})$. The $i$%
-th component $h_{i}$ of the $h$-vector of $S_{G}/\bar{Q}_{G}$ and,
consequently, of $\bar{A}(G)$ is equal to the number of linear extensions of 
$P_{n}$, which, seen as permutations of $[n]$, have exactly $i$ descents (%
\cite{ReiWel}). In particular, 
\begin{equation}\label{equh}
h_{i}\geq 0, \text{\ for all\ } 0\leq i\leq n-1,\ h_0=1, \text{ and }h_{n}=h_{n+1}=0. 
\end{equation}
}

\emph{For example, if $P_{n}^{\prime \prime }=\{p_{1}^{\prime \prime
},...,p_{n}^{\prime \prime }\}$ is an antichain, then each permutation of $%
[n]$ can be seen as a linear extension of $P_{n}^{\prime \prime }$, hence,
for all $0\leq i\leq n-1 $, the $i$-th component of the $h$-vector of $%
\Delta(\mathcal{J}(P_n^{\prime\prime}))$ is equal to the number of all
permutations of $[n]$ with exactly $i$ descents, which is the Eulerian
number $A(n,i)$. }
\end{Remark}

For each $\emptyset \neq F\subset \lbrack n]$ we denote by $P_{n}(F)$ the
subposet of $P_{n}$ induced by the subset $\{p_{i}|i\in F\}$. The main
result of the paper relates the Hilbert series of $A(G)$ to the Hilbert
series of $\bar{A}(G_{F})$, for all $F\subset \lbrack n]$, where $G_{F}$
denotes the bipartite graph that comes from the poset $P_{n}(F)$. If $%
F=\emptyset $, then, by convention, the Hilbert series of $\bar{A}(G_{F})$
is equal to $\frac{1}{1-z}$.\\

In order to prove the main theorem we need a preparatory result. \\

Let $\emptyset\neq F\subsetneq [n]$ and let $\alpha$ be a poset ideal of $P_n(\bar{F}),$ where by $\bar{F}$ we mean the complement of $F$ 
in $[n]$. We denote by $\delta_\alpha$ the maximal subset 
of $P_n(F)$ such that $\alpha\cup\delta_\alpha\in \mathcal{J}(P_n)$. Note that 
\[
\delta_\alpha=\cup\{\gamma \ |\ \gamma\subset P_n(F),  \alpha\cup \gamma\in \mathcal{J}(P_n)\}.
\]
If we set $\beta=\alpha\cup \delta_\alpha,$ then, by the definition of $\delta_\alpha,$ $\beta$ has the following property: for any 
$j\in F,$ $p_j\not \in \beta$ implies 
 $\beta\cup\{p_j\}\not\in \mathcal{J}(P_n).$

\begin{Lemma}\label{lemadelta} Let $\emptyset\neq F\subsetneq [n]$ and 
let $\mathcal{S}$ be the set of poset ideals $\beta$ of $P_n$ with the property that for any $j\in F$ such that $p_j\not\in \beta$ we have 
$\beta\cup\{p_j\}\not\in \mathcal{J}(P_n).$ Then the map $\varphi\colon \mathcal{J}(P_n(\bar{F}))\rightarrow \mathcal{S}$ defined 
by $\alpha\mapsto \beta:= \alpha\cup\delta_{\alpha}$, is an isomorphism of posets.
\end{Lemma}

\begin{proof}
$\varphi$ is invertible. Indeed, the map $\psi:\mathcal{S}\rightarrow \mathcal{J}(P_n(\bar{F}))$ defined by 
$\psi(\beta)=\beta\cap P_n(\bar{F})$ is the inverse of $\varphi$ since if $\alpha=\beta\cap P_n(\bar{F}),$ then, by the property of 
$\beta,$ we have $\delta_{\alpha}=\beta\setminus P_n(\bar{F}).$

Let $\alpha_1\subsetneq \alpha_2$ be poset ideals of $P_n(\bar{F})$ and $\beta_i=\varphi(\alpha_i)=\alpha_i\cup\delta_i,\ i=1,2.$ We 
only need 
to show that $\beta_1\subset\beta_2$ since the strict inclusion follows from the hypothesis $\alpha_1\subsetneq \alpha_2$. Let us 
assume that $\beta_1\not\subset \beta_2$ and let $p_a, a\in F,$ be a minimal element in $\beta_1\setminus \beta_2.$ Since 
$p_a\not\in \beta_2,$ it follows that $\beta_2\cup \{p_2\}$ is not a poset ideal of $P_n.$ Therefore there exists $p_b<p_a$ such that 
$p_b\not\in\beta_2.$ On the other hand, $p_b\in\beta_1$ since $\beta_1\in\mathcal{J}(P_n),$ hence, $p_b\in\beta_1\setminus\beta_2,$
which leads to a contradiction with the choice of $p_a.$ 

Now let $\beta_1\subsetneq \beta_2,$ $\beta_1,\beta_2\in\mathcal{S},$ and assume that 
$\alpha_1=\alpha_2,$ where $\alpha_1=\beta_1\cap P_n(\bar{F}),$ and $\alpha_2=\beta_2\cap P_n(\bar{F}).$ Then 
$\delta_1=\beta_1\setminus P_n(\bar{F})\subsetneq \delta_2=\beta_2\setminus P_n(\bar{F}).$ But this is impossible since 
$\delta_1$ is  maximal among the subsets $\gamma$ of  $P_n(F)$ such that $\alpha_1\cup\gamma\in\mathcal{J}(P_n).$
\end{proof}

We can state now the main theorem which relates the Hilbert series of the vertex cover algebra $A(G)$ to the Hilbert series of 
the basic cover algebras $\bar{A}(G_F)$ for all $F\subset [n].$

\begin{Theorem}
\label{comp} For $F\subset \lbrack n]$ let $H_{\bar{A}(G_{F})}(z)$ be the
Hilbert series of $\bar{A}(G_{F})$ and let $H_{A(G)}(z)$ be the Hilbert
series of $A(G)$. Then:

\begin{equation}
H_{A(G)}(z)=\frac{1}{(1-z)^{n}}\sum\limits_{F\subset \lbrack n]}H_{\bar{A}%
(G_{F})}(z)\left( \frac{z}{1-z}\right) ^{n-\left\vert F\right\vert }\text{.}
\label{equ6}
\end{equation}

In particular, if $h(z)=\sum\limits_{j\geq 0}h_{j}z^{j}$ and $%
h^{F}(z)=\sum\limits_{j\geq 0}h_{j}^{F}z^{j}$, where $h=(h_{j})_{j\geq 0}$
and $h^{F}=(h_{j}^{F})_{j\geq 0}$ are the $h$-vectors of $A(G)$, and,
respectively, $\bar{A}(G_{F})$, then 
\begin{equation}
h(z)=\sum\limits_{F\subset [n]}h^{F}(z) z^{n-\left\vert
F\right\vert }\text{.}  \label{equ7}
\end{equation}
\end{Theorem}

\begin{proof}
Let $J_{G}=\ini_{<_{lex}^{\#}}(Q_{G})$. It is known that $B_{G}/Q_{G}$ and $%
B_{G}/J_{G}$ have the same Hilbert series. Let $B_{G}^{\prime
}=K[\{x_{i}\}_{1\leq i\leq n},\{u_{\alpha }\}_{\alpha \in \mathcal{J}%
(P_{n})}]$. By using the following $K$-vector space isomorphism%
\begin{equation*}
B_{G}/J_{G}\simeq K[y_{1},y_{2},...,y_{n}]\otimes _{K}B_{G}^{\prime
}/(J_{G}\cap B_{G}^{\prime })\text{,}
\end{equation*}

\noindent we get 
\[H_{A(G)}(z)=H_{B_{G}/Q_{G}}(z)=H_{B_{G}/J_{G}}(z)=\frac{1}{(1-z)^{n}}%
H_{B_{G}^{\prime }/(J_{G}\cap B_{G}^{\prime })}(z).
\]

We need to compute the Hilbert series of $B_{G}^{\prime }/(J_{G}\cap
B_{G}^{\prime })$. To this aim we show that we have an isomorphism of  $K$-vector spaces 
\begin{equation}
B_{G}^{\prime }/(J_{G}\cap B_{G}^{\prime })\simeq\bigoplus\limits_{F\subset
\lbrack n]}\bar{A}(G_{\bar{F}})\otimes _{K}x_{F}K[\{x_{i}\}_{i\in F}]%
\text{.}  \label{equ8}
\end{equation}

For $\emptyset \neq F\subset [n]$ let $J_{F}$ be the initial ideal with respect to $<^{\#}$ of the toric ideal of $\bar{A}(G_F).$
Then $J_{F}=(u_{\alpha }u_{\beta }|\alpha ,\beta \in \mathcal{J}%
(P_{n}(F)),\alpha \not\subset \beta ,\beta \not\subset \alpha )$. If $%
F=\emptyset $, we put by convention $J_{F}=(u_{\emptyset })$.

The basic vertex cover algebra $\bar{A}(G_{\bar{F}})$ can be decomposed
as a $K$-vector space as $\bar{A}(G_{\bar{F}})\simeq\bigoplus\limits_{w%
\notin J_{\bar{F}}}Kw$. We notice that $w\notin J_{\bar{F}}$ if
and only if supp$(w)=\{\alpha _{1},...,\alpha _{s}\}$, $s\geq 0$, where $%
\alpha _{1}\subsetneq\ldots \subsetneq \alpha _{s}$ is a chain in $\mathcal{J}(P_{n}(\bar{F}%
))$. It follows that for $F\subset \lbrack n]$ we have%
\begin{equation*}
V_F:=\bar{A}(G_{\bar{F}})\otimes _{K}x_{F}K[\{x_{i}\}_{i\in F}]\simeq
\bigoplus K v w\text{,}
\end{equation*}

\noindent where the direct sum is taken over all monomials $vw$ with $v$ 
monomial in the variables $x_{i}$ such that supp$(v)=F$ and $w$  monomial
in the variables $u_{\alpha }$ such that $w\notin J_{\bar{F}}$.
As a $K$-vector space, 
$B_{G}^{\prime }/(J_{G}\cap B_{G}^{\prime })$ has the 
decomposition
\[
B_{G}^{\prime }/(J_{G}\cap B_{G}^{\prime })\simeq\bigoplus\limits_{F\subset \lbrack n]}\bigoplus W_F,
\]
where $W_F=\bigoplus K v w^{\prime }$ and the direct sum is taken over all monomials $v$ with 
supp$(v)=F$ and all monomials  $w^{\prime }$ in the variables $%
u_{\alpha }$ with $\alpha \in \mathcal{J}(P_{n})$ such that $v w^{\prime }\neq 0$ modulo $J_{G}\cap B^{\prime}_G$.

In order to prove (\ref{equ8}), we only need to show that for each $F\subset [n],$ the $K$-vector spaces $V_F$ and $W_F$ are
isomorphic. This is obvious for $F=\emptyset$ and $F=[n].$

Let us consider now $\emptyset\neq F\subsetneq [n].$ Based on the previous lemma, we are going to show that there exists a bijection
between the $K$-bases of $V_F$ and $W_F.$

Let $v w$ be an element of the $K$-basis of $V_F.$ This means that $\supp(v)=F$ and $w$ is of the form 
$w=u_{\alpha_1}^{a_1}\cdots u_{\alpha_s}^{a_s}$ for some chain $\alpha_1\subsetneq \ldots \subsetneq \alpha_s$ in 
$\mathcal{J}(P_n),$ $s\geq 1.$ For each $1\leq i\leq s,$ let $\beta_i=\varphi(\alpha_i)\in \mathcal{J}(P_n)$ as it was defined in Lemma \ref{lemadelta}.
We map $v w$ to the monomial $v w^{\prime}$ where $w^{\prime}=u_{\beta_1}^{a_1}\cdots u_{\beta_s}^{a_s}.$ By Lemma \ref{lemadelta}, 
we have that $\beta_1\subsetneq \ldots \subsetneq \beta_s$ is a chain in $\mathcal{J}(P_n)$. Moreover, for any $j\in F$ and any 
$\beta_i$ such that $p_j\not\in \beta_i,$ we have $\beta_i\cup\{p_j\}\not\in \mathcal{J}(P_n).$ Therefore, $v w^{\prime}$ is a monomial
in the $K$-basis of $W_F.$ 

Conversely, let  $v w^{\prime}$ be a monomial from the $K$-basis of $W_F$, where $\supp(v)=F$ and 
$w^{\prime}=u_{\beta_1}^{a_1}\cdots u_{\beta_s}^{a_s},$ with $\beta_1\subsetneq \ldots \subsetneq \beta_s$  a chain in 
$\mathcal{J}(P_n).$ Let $\alpha_i=\beta_i\cap P_n(\bar{F}),$ for $1\leq i\leq s.$ Then  we associate to $v w^{\prime}$ the monomial $v w$ in the $K$-basis of $V_F$, where $w=u_{\alpha_1}^{a_1}\cdots u_{\alpha_s}^{a_s}.$

By using again Lemma \ref{lemadelta} it follows that the above defined maps between the $K$-bases of $V_F$ and $W_F$ are inverse.

By (\ref{equ8}) we get
\[
H_{B_{G}^{\prime }/(J_{G}\cap B_{G}^{\prime })}(z)=\sum\limits_{F\subset
\lbrack n]}H_{\bar{A}(G_{\bar{F}})}(z)\left( \frac{z}{1-z}\right)
^{\left\vert F\right\vert }=\sum\limits_{F\subset \lbrack n]}H_{\bar{A}%
(G_{F})}(z)\left( \frac{z}{1-z}\right) ^{n-\left\vert F\right\vert }.
\]
 Hence 
\[
H_{A(G)}(z)=\frac{1}{(1-z)^{n}}\sum\limits_{F\subset \lbrack n]}H_{\bar{A}%
(G_{F})}(z)\left( \frac{z}{1-z}\right) ^{n-\left\vert F\right\vert }.
\] 
Since 
$H_{A(G)}(z)=\frac{h(z)}{(1-z)^{2n+1}}$
and $H_{\bar{A}(G_{F})}=\frac{%
h^{F}(z)}{(1-z)^{n+1}}$, for all $F\subset \lbrack n]$, it follows that $%
h(z)=\sum\limits_{F\subset \lbrack n]}h^{F}(z) z^{n-\left\vert
F\right\vert }$.
\end{proof}

\begin{Corollary}
\label{hvect} For all $0\leq j\leq n-1$, the $j $-th component $h_{j}$ of
the $h$-vector of $A(G)$ is equal to the number of all linear extensions of
all $n-l$-element subposets of $P_{n} $, which, seen as permutations of $%
[n-l]$, have exactly $j-l$ descents, for all $0\leq l\leq j.$
\end{Corollary}

\begin{proof}
It follows immediately from (\ref{equ7}) and Remark \ref{euler}.
\end{proof}

\begin{Corollary}
The $h$-vector of $A(G)$ is unimodal.
\end{Corollary}

\begin{proof}
By (\ref{equ7}) we get $h_{n+1}=\sum\limits_{F\subset \lbrack
n]}h_{\left\vert F\right\vert +1}^{F}$ and $h_{n}=\sum\limits_{F\subset
\lbrack n]}h_{\left\vert F\right\vert }^{F}$. By using (\ref{equh}) from
Remark \ref{euler}, we have $h_{\left\vert F\right\vert }^{F}=h_{\left\vert
F\right\vert +1}^{F}=0$, for all $\emptyset \neq F\subset \lbrack n]$. Hence 
$h_{n+1}=h_{1}^{\emptyset }=0$ and $h_{n}=h_{0}^{\emptyset }=1$. In \cite[%
Corollary 4.4]{HerHib3} it is proved that $A(G)$ is a Gorenstein ring,
hence, by \cite[Corollary 4.3.8 (b) and Remark 4.3.9 (a)]{BruHer}, $%
h_{i}=h_{n-i}$, for all $0\leq i\leq n$. We denote by $\nu (l,j)$ the number
of all linear extensions of all $n-l$-element subposets of $P_{n}$ which,
seen as permutations of $[n-l]$, have exactly $j-l$ descents. Hence, by
Corollary \ref{hvect}, $h_{j}=\sum\limits_{l=0}^{j}\nu (l,j)$. Let $0\leq
j<j+1\leq \lfloor \frac{n}{2}\rfloor $. Then $\nu (l,j)\leq \nu (l+1,j+1)$,
for all $0\leq l\leq j$, which implies that $h_{j+1}=\nu
(j+1,0)+\sum\limits_{l=0}^{j}\nu (j+1,l+1)\geq \sum\limits_{l=0}^{j}\nu
(l,j)=h_{j}$.
\end{proof}

\begin{Remark}
\emph{\label{hilser} The Hilbert series of the vertex cover algebra $A(G)$
is given by 
\begin{equation*}
H_{A(G)}(z)=\frac{h_{0}+h_{1}z+...+h_{n-1}z^{n-1}+h_{n}z^{n}}{(1-z)^{2n+1}},
\end{equation*}%
where $h=(h_{0},\ldots ,h_{n})$ is the $h$-vector of $A(G)$. In particular,
we recover the known fact that $\dim A(G)=2n+1$. It also follows that the $a$%
-invariant is $a=-n-1.$}
\end{Remark}

\begin{Corollary}
\label{cormult} Let $e(A(G))$ be the multiplicity of $A(G)$ and let $e(\bar{A%
}(G_{F}))$  the multiplicity of $\bar{A}(G_{F})$ for $F\subset \lbrack n]$%
. Then 
\begin{equation*}
e(A(G))=\sum\limits_{F\subset \lbrack n]}e(\bar{A}(G_{F})).
\end{equation*}
\end{Corollary}

\begin{proof}
It follows immediately from (\ref{equ7}).
\end{proof}

Let $P_{3}=\{p_{1},p_{2},p_{3}\}$ be the poset with $p_{1}\leq p_{2}$ and $%
p_{1}\leq p_{3}$ and $G_{3}=G(P_{3})$. Then $H_{\bar{A}(G_{\emptyset })}(z)=%
\frac{1}{1-z}$, $H_{\bar{A}(G_{\{1\}})}(z)=H_{\bar{A}(G_{\{2\}})}(z)=H_{\bar{%
A}(G_{\{3\}})}(z)=\frac{1}{(1-z)^{2}}$, $H_{\bar{A}(G_{\{1,2\}})}(z)=H_{\bar{%
A}(G_{\{1,3\}})}(z)=\frac{1}{(1-z)^{3}}$, $H_{\bar{A}(G_{\{2,3\}})}(z)=\frac{%
1+z}{(1-z)^{3}}$, $H_{\bar{A}(G_{\{1,2,3\}})}(z)=\frac{1+z}{(1-z)^{4}}$ and
the Hilbert series of $A(G_{3})$ is: 
\begin{equation*}
H_{A(G_{3})}(z)=\frac{1}{(1-z)^3}\sum\limits_{F\subset \lbrack 3]}H_{\bar{A}%
(G_{F})}(z)\left( \frac{z}{1-z}\right) ^{3-\left\vert F\right\vert }=\frac{%
z^{3}+4z^{2}+4z+1}{(1-z)^{7}}.
\end{equation*}

Hence $h_{0}=h_{3}=1$, $h_{1}=h_{2}=4$, $h_{4}=0$, $e(A(G_{3}))=10$. We can
also compute the $h$-vector of $A(G_{3})$ by using Corollary \ref{hvect}. The poset $%
P_{3}$ has two linear extensions, which, seen as permutation of $[3]$, are
equal to $id_{3}$ and $(23)$. Hence $h_{0}=1$, since there exists only one
linear extension of $P_{3}$, which, seen as a permutation of $[3]$, has
exactly $0$ descents. Furthermore, $P_{3}$ has three $2$-element subposets,
the chains $P_{3}(\{1,2\})$ and $P_{3}(\{1,3\})$ with a linear extension
corresponding to $id_{2}$, and the antichain $P_{3}(\{2,3\})$ with two
linear extensions corresponding to $id_{2}$ and $(12)$. Thus $h_{1}=4$,
since there exists only one linear extension of $P_{3}$, which, seen as a
permutation of $[3]$, has exactly $1$ descent and each of the subposets $%
P_{3}(\{1,2\})$, $P_{3}(\{1,3\})$ and $P_{3}(\{2,3\})$ has one linear
extension, which, seen as a permutation of $[2]$, has exactly $0$ descents.%
\newline

Let $\mathcal{L}_{n}$ be the Boolean lattice on $\{p_{1},p_{2},...,p_{n}\}$, 
$n\geq 1$, and $A(p,q)$ be the Eulerian number for $1\leq q\leq n$ and $%
0\leq p<q$. By convention, we put $A(0,0)=1$ and $A(q,q)=0$, for all $1\leq
q\leq n$.

We compute the Hilbert series of the vertex cover algebra of the
Cohen-Macaulay bipartite graphs that come from a chain and an antichain.

\begin{Proposition}
\label{capete} Let $G^{\prime}$ be a bipartite graph that comes from a chain
and $G^{\prime \prime }$ a bipartite graph that comes from an antichain with 
$n$ elements, $n\geq 1$. Then we have

\begin{itemize}
\item[(i)] $H_{A(G^{\prime })}(z)=\frac{(1+z)^{n}}{(1-z)^{2n+1}}. $ In
particular, $e(A(G^{\prime }))=2^{n}$.

\item[(ii)] $H_{A(G^{\prime \prime})}(z)=\frac{\sum\limits_{j=0}^{n}\sum%
\limits_{l=0}^j\binom{n}{l}A(n-l,j-l)z^j}{(1-z)^{2n+1}}. $ In particular, $%
e(A(G^{\prime \prime }))=n!\cdot \sum\limits_{l=0}^{n}\frac{1}{l!}$.
\end{itemize}
\end{Proposition}

\begin{proof}
(i) We may assume that $G^{\prime }=G(P_{n}^{\prime })$, where $%
P_{n}^{\prime }=\{p_{1}^{\prime },p_{2}^{\prime },...,p_{n}^{\prime }\}$ is
the chain with $p_{1}^{\prime }\leq p_{2}^{\prime }\leq ...\leq
p_{n}^{\prime }$. $P_n^{\prime}$ as well as all its subposets have a unique
linear extension. Therefore, the $h$-vector of $G^{\prime}$ is ${\binom{n }{0%
}}, {\binom{n }{1}},\ldots, {\binom{n}{n}}.$

(ii) Let $G^{\prime \prime }=G(P_{n}^{\prime \prime })$, where $%
P_{n}^{\prime \prime }=\{p_{1}^{\prime \prime },...,p_{n}^{\prime \prime }\}$
is an antichain. If $F=[n]$, then, by convention, $A(0,0)=1=h_{0}^{\bar{%
F}}$. If $F\varsubsetneqq \lbrack n]$, then $\mathcal{J}(P_{n}^{\prime
\prime }(\bar{F}))$ is a Boolean lattice on the set $P_{n}^{\prime
\prime }(\bar{F})$, which implies that $\mathcal{J}(P_{n}^{\prime
\prime }(\bar{F}))$ is isomorphic to $\mathcal{L}_{n-l}$, where $%
l=\left\vert F\right\vert $. Therefore, by Remark \ref{euler}, $h_{i}^{%
\bar{F}}=A(n-l,i)$, for all $0\leq i\leq n-l-1$. If $i=n-l$, then $%
A(n-l,i)=0$ (by convention) and $h_{i}^{\bar{F}}=0$ (by Remark \ref%
{euler}), which implies that $A(n-l,i)=h_{i}^{\bar{F}}$. By (\ref{equ6}%
) we have $h_{j}^{\prime \prime }=\sum\limits_{l=0}^{j}\sum\limits 
_{\substack{ F\subset \lbrack n]  \\ \left\vert F\right\vert =l}}h_{j-l}^{%
\bar{F}}$, hence $h_{j}^{\prime \prime }=\sum\limits_{l=0}^{j}\binom{n}{%
l}A(n-l,j-l),$ for all $0\leq j\leq n$.

We get $e(A(G^{\prime \prime }))=\sum\limits_{j=0}^{n}h_{j}^{\prime \prime
}=\sum\limits_{j=0}^{n-1}h_{j}^{\prime \prime }+1
=\sum\limits_{j=0}^{n-1}\sum\limits_{l=0}^{j}\binom{n}{l}A(n-l,j-l)+1=\sum%
\limits_{l=0}^{n-1}\binom{n}{l}$.$\sum\limits_{j=0}^{n-l-1}A(n-l,j)+1$. We
obviously have $\sum\limits_{j=0}^{n-l-1}A(n-l,j)=(n-l)!$, for all $0\leq
l\leq n-1$. Therefore, $e(A(G^{\prime \prime }))=\sum\limits_{l=0}^{n-1}%
\binom{n}{l}\cdot (n-l)!+1=n!\cdot \sum\limits_{l=0}^{n}\frac{1}{l!}$.
\end{proof}

\begin{Remark}
\emph{The reduced Gr\"{o}bner basis $\mathcal{G}^{\prime }$ of the toric
ideal $Q_{G^{\prime }}$ of $A(G^{\prime })$ with respect to the monomial
order $<_{lex}^{\sharp }$ on the polynomial ring $B_{G^{\prime }}$ is: 
\begin{equation*}
\mathcal{G}^{\prime }=\{\underline{x_{j}u_{\{p_{1}^{\prime },\ldots
,p_{j-1}^{\prime }\}}}-y_{j}u_{\{p_{1}^{\prime },...,p_{j}^{\prime }\}}|j\in
\lbrack n]\},
\end{equation*}%
where the initial monomial of each binomial of $\mathcal{G}^{\prime }$ is
the first monomial. }

\emph{We notice that the initial ideal $\ini_{<_{lex}^{\sharp
}}(Q_{G^{\prime }})=(x_{j}u_{\{p_{1}^{\prime },\ldots ,p_{j-1}^{\prime
}\}}|j\in \lbrack n])$ is a complete intersection, which implies that the
toric ideal $Q_{G^{\prime }}$ is a complete intersection. Thus $A(G^{\prime
})$ has a pure resolution given by the Koszul complex. }
\end{Remark}

\begin{Proposition}
\label{uni} Let $G$ be a Cohen-Macaulay bipartite graph on $V_{n}$, $n\geq 1$%
. Then the following assertions hold:

\begin{itemize}
\item[(i)] $G$ comes from a chain if and only if \ $H_{A(G)}(z)=\frac{%
(1+z)^{n}}{(1-z)^{2n+1}}$;

\item[(ii)] $G$ comes from an antichain if and only if $H_{A(G)}(z)=\frac{%
h_{n}^{\prime \prime }z^{n}+h_{n-1}^{\prime \prime }z^{n-1}\ldots
+h_{1}^{\prime \prime }z+h_{0}^{\prime \prime }}{(1-z)^{2n+1}}$, where $%
h^{\prime \prime }=(h_{0}^{\prime \prime },h_{1}^{\prime \prime },\ldots
,h_{n}^{\prime \prime })$ is the $h$-vector of the vertex cover algebra $%
A(G^{\prime \prime })$ of the bipartite graph $G^{\prime \prime }$ that
comes from an antichain $P_{n}^{\prime \prime }=\{p_{1}^{\prime \prime
},p_{2}^{\prime \prime },\ldots ,p_{n}^{\prime \prime }\}$.
\end{itemize}
\end{Proposition}

\begin{proof} Let us suppose that $G$ comes from a poset $P_{n}=%
\{p_{1},p_{2},...,p_{n}\}$, $n\geq 1$, and let $h=(h_{0},h_{1},...,h_{n})$
be the $h$-vector of $A(G)$. In the first place we need to compute the component $h_1$. 
By using (\ref{equ7}), we get $h_1=h_1^{[n]}+n.$ But $h_1^{[n]}$ is the component of rank $1$ in the 
$h$-vector of $\bar{A}(G).$ By using the formula which relates the $h$-vector to the $f$-vector for the order complex $\Delta(\mathcal{J}(P_n)),$ 
we immediately get $h_1^{[n]}=|\mathcal{J}(P_n)|-n-1,$ which implies that $h_1=|\mathcal{J}(P_n)|-1.$\\
(i) Let  $h_{1}=n$. Then $%
\left\vert \mathcal{J}(P_{n})\right\vert =n+1$, which implies that $P_{n}$
is a chain.\\
(ii) Let $h_{1}=h_{1}^{\prime \prime }=$ $%
\left\vert \mathcal{J}(P_{n}^{\prime \prime })\right\vert -1=2^{n}-1$. Then $\left\vert \mathcal{J}(P_{n})\right\vert
=2^{n}$, which implies that $P_{n}$ is an antichain.

In both cases the converse follows from Proposition \ref{capete}.
\end{proof}

\begin{Proposition}
\label{hine} Let $G$ be a Cohen-Macaulay bipartite graph on $V_{n}$, $n\geq
1 $. If $h=(h_{0},h_{1},...,h_{n})$ is the $h$-vector of $A(G)$, then $%
\binom{n}{j}\leq h_{j}\leq h_{j}^{\prime \prime }$, for all $0\leq j\leq n$,
where $G^{\prime \prime }$ comes from an antichain with $n$ elements and $%
h^{\prime \prime }=(h_{0}^{\prime \prime },h_{1}^{\prime \prime
},...,h_{n}^{\prime \prime })$ is the $h$-vector of $A(G^{\prime \prime })$.
\end{Proposition}

\begin{proof}
We may assume without loss of generality that $G=G(P_{n})$, where $%
P_{n}=\{p_{1},...,p_{n}\}$ is a poset such that $p_{i}\leq p_{j}$ implies $%
i\leq j$. Let $P_{n}^{\prime }=\{p_{1}^{\prime },p_{2}^{\prime
},...,p_{n}^{\prime }\}$ the chain with $p_{1}^{\prime }\leq p_{2}^{\prime
}\leq ...\leq p_{n}^{\prime }$ and $P_{n}^{\prime \prime }=\{p_{1}^{\prime
\prime },p_{2}^{\prime \prime },...,p_{n}^{\prime \prime }\}$ an antichain.
By using (\ref{equ7}) and (\ref{equh}), we get $h_{0}=1=h_{0}^{\prime \prime }$ and $%
h_{n}=1=h_{n}^{\prime \prime }.$ Let $1\leq j\leq n-1$. By 
Corollary \ref{hvect}, $h_{j}$ is equal to the number of  all
linear extensions of all $n-l$-element subposets, which, seen as
permutations of $[n-l]$, have exactly $j-l$ descents, for all $0\leq l\leq j$%
. Each $n-l$-element subposet of $P_{n}^{\prime }$, respectively, $%
P_{n}^{\prime \prime }$ is a chain, respectively, an antichain, hence it has
only one linear extension which corresponds to $id_{n-l}$, respectively, it
has $(n-l)!$ linear extensions which correspond to all permutations of $%
[n-l] $. Therefore $\binom{n}{j}\leq h_{j}\leq h_{j}^{\prime \prime }$, for
all $1\leq j\leq n-1$.
\end{proof}

\begin{Corollary}
\label{ineq} Let $G$ be a bipartite graph that comes from a poset with $n$
elements, $n\geq 1$. Then $2^{n}\leq e(A(G))\leq n!\sum\limits_{l=0}^{n}%
\frac{1}{l!}$. The left equality holds if and only if the poset is a chain
and the right equality holds if and only if the poset is an antichain.
\end{Corollary}

\begin{proof}
Let $G^{\prime }=G(P_{n}^{\prime })$ and $G^{\prime \prime }=G(P_{n}^{\prime
\prime })$, where $P_{n}^{\prime }=\{p_{1}^{\prime },p_{2}^{\prime },\ldots
,p_{n}^{\prime }\}$ is a chain and $P_{n}^{\prime \prime }=\{p_{1}^{\prime
\prime },p_{2}^{\prime \prime },\ldots ,p_{n}^{\prime \prime }\}$ is an
antichain. We may assume without loss of generality that $p_{1}^{\prime
}\leq p_{2}^{\prime }\leq ...\leq p_{n}^{\prime }$ and $G=G(P_{n})$, where $%
P_{n}=\{p_{1},p_{2},...,p_{n}\}$ is a poset such that $p_{i}\leq p_{j}$
implies $i\leq j$. Let $h$, $h^{\prime }$, respectively, $h^{\prime \prime }$
be the $h$-vector of $A(G)$, $A(G^{\prime })$, respectively, $A(G^{\prime
\prime })$. By summing up the inequalities $h_{j}^{\prime }\leq h_{j}\leq
h_{j}^{\prime \prime }$ from Proposition \ref{hine} or by applying Corollary %
\ref{multineq}, we obtain $e(A(G^{\prime }))\leq e(A(G))\leq e(A(G^{\prime
\prime }))$. Next, from Proposition \ref{capete}, we get the desired
inequalities. The left equality, respectively, the right equality holds if
and only if $h_{j}^{\prime }=h_{j}$, respectively, $h_{j}=h_{j}^{\prime
\prime }$, for all $0\leq j\leq n$, therefore, by using Proposition \ref{uni}%
, this is equivalent to $P_{n}=P_{n}^{\prime }$, respectively, $%
P_{n}=P_{n}^{\prime \prime }$.
\end{proof}

\section*{Acknowledgment}

I would like to  thank Professor J\"{u}rgen Herzog for very useful
suggestions and discussions on the subject of this paper. I am also very
grateful to Professor Volkmar Welker who explained to me the combinatorial
significance of the $h$-vector of a Hibi ring.

\end{document}